\documentclass[12pt,reqno]{amsart}

\usepackage{amsfonts,amsmath,amssymb,amsthm}
\usepackage{latexsym}
\usepackage{mathrsfs}
\usepackage{enumitem}
\usepackage{nicefrac}
\usepackage{accents}

\usepackage{amsaddr}

\makeatletter
\@namedef{subjclassname@2010}{%
  \textup{2010} Mathematics Subject Classification}
\makeatother

\newtheorem{Theorem}{Theorem}[section]
\newtheorem{Lemma}{Lemma}[section]
\newtheorem{Proposition}{Proposition}[section]
\newtheorem{Corollary}{Corollary}[section]
\newtheorem{Remark}{Remark}[section]
\newtheorem{Definition}{Definition}[section]

\numberwithin{equation}{section}

\newcommand{\funcion}[5]{%
{\setlength{\arraycolsep}{2pt}
	\begin{array}{r@{}ccl}
	#1\colon & #2 & \longmapsto & #3\\
	& #4 & \longmapsto & #5
\end{array}}}

\begin{document}

\baselineskip=17pt

\title[A primitive associated to the Cantor-Bendixson derivative]
{A primitive associated to the Cantor-Bendixson derivative on the real line}

\author[B. \'Alvarez-Samaniego]{Borys \'Alvarez-Samaniego}
\address{\vspace{-3mm} N\'ucleo de Investigadores Cient\'{\i}ficos\\
	Facultad de Ingenier\'{\i}a, Ciencias F\'{\i}sicas y Matem\'atica\\
	Universidad Central del Ecuador (UCE)\\
	Quito, Ecuador}
\email{balvarez@uce.edu.ec, balvarez@impa.br}

\author[A. Merino]{Andr\'es Merino}
\address{\vspace{-3mm} Facultad de Ciencias\\
	Escuela Polit\'ecnica Nacional (EPN)\\
	Quito, Ecuador}
\email{andres.merino@epn.edu.ec}

\date{May 02, 2016.}

\begin{abstract}
We consider the class of compact countable subsets of the real numbers $\mathbb{R}$. 
By using an appropriate partition, up to homeomorphism, of this class we give a 
detailed proof of a result shown by S. Mazurkiewicz and W. Sierpinski related to   
the cardinality of this partition. Furthermore, for any compact subset of 
$\mathbb{R}$, we show the existence of a ``primitive'' 
related to its Cantor-Bendixson derivative.
\end{abstract}


\maketitle


\section{Introduction}
The earliest ideas of limit point and derived set in the space of the real numbers 
were both introduced and investigated by Georg Cantor since 1872 
(see also \cite{Cantor1872, Cantor1879, Cantor1880, Cantor1882, Cantor1883}) 
to analyze the convergence set of a trigonometric series.  These two concepts 
have been generalized to the case of any arbitrary topological space.  Thus,  let $X$ 
be a topological space and let $A$ be a subset of $X$, we write $A'$ to denote 
the derived set of $A$, that is, the set of all limit points of $A$.  The next 
definition extends the process of taking the derivative of a set for any ordinal
number. 

\begin{Definition}[Cantor-Bendixson's derivative]\label{def:D_CB}
Let $A$ be a subset of a topological space. For a given ordinal number $\alpha$, we 
define, using Transfinite Recursion, the  \emph{$\alpha$-th derivative} of $A$, 
written $A^{(\alpha)}$, as follows:
\begin{itemize}
  \item $A^{(0)}=A$,
  \item $A^{(\beta+1)}=(A^{(\beta)})'$, for all ordinal $\beta$,
  \item $\displaystyle A^{(\lambda)}=\bigcap_{\gamma<\lambda} A^{(\gamma)}$, 
	for all limit ordinal $\lambda\neq 0$.
\end{itemize}
\end{Definition}

In this paper, we are initially concerned with the Cantor-Bendixson 
derivative of compact countable subsets of the real numbers, where 
a countable set is either a finite set or a countably infinite set.  
Thus, we consider the set
\begin{equation}\label{setK}
    \mathcal{K}=\{K\subset \mathbb{R}:K\text{ is compact and countable}\}.
\end{equation}
Moreover, for all $K_1, K_2\in\mathcal{K}$, we define the relation
\begin{equation}\label{relation1}
    K_1 \sim K_2 \Longleftrightarrow 
		\text{ there exists } f\colon K_1 \longmapsto K_2\text{ continuous and bijective.}
\end{equation}
It is not hard to see that $\sim$ is an equivalence relation on the set $\mathcal{K}$ 
and since the elements of $\mathcal{K}$ are compact sets, we have that for all 
$K_1, K_2\in\mathcal{K}$
\begin{equation}\label{relation2}
    K_1 \sim K_2 \Longleftrightarrow 
		\text{ there exists } f\colon K_1 \longmapsto K_2\text{ homeomorphism.}
\end{equation}
Therefore, there is a partition of the set $\mathcal{K}$, and we denote by 
\begin{equation}\label{partition}
    \mathscr{K}= {\Large\nicefrac{\mathcal{K}}{\sim}}
\end{equation}
the set of all equivalence classes of $\mathcal{K}$.

In 1920, S. Mazurkiewicz and W. Sierpinski \cite{Sierpinski} showed that the 
cardinality of $\mathscr{K}$ is $\aleph_1$.  In Section 2, we show in detail  
that for any countable ordinal number $\alpha$, and 
for any $p \in \omega$, there is a set $K \in \mathcal{K}$ such that 
$K^{(\alpha)}$ has exactly $p$ elements.  This last fact was first 
briefly mentioned by Cantor in \cite{Cantor1880}.  The results shown 
in Section 2 allow us to prove, in Theorem \ref{Thcardinality}, that the 
cardinality of $\mathscr{K}$ is greater than or equal to $\aleph_1$.  
On the other hand, the cardinality of $\mathscr{K}$ is smaller than or equal to 
$\aleph_1$ as a consequence of Theorem \ref{Th3}.

Section 3 considers Cantor-Bendixson's characteristic, denoted by $\mathcal{CB}$.  
First, we show that for any element $K \in \mathcal{K}$ with 
$\mathcal{CB}(K)=(\alpha,p)$, we get $p=0$ if and only if $K=\varnothing$. Moreover,
we use Lemma \ref{Lem5} to prove Theorem \ref{Th3}, where the injectivity 
of function $\mathcal{\widetilde{CB}}$, defined in \eqref{eq:CBtilde}, is shown.   
These two last results were first mentioned in \cite{Sierpinski}; however, for 
the sake of completeness, we include here their detailed proofs.  Finally,  
Theorem \ref{Thprimitive} shows that for any compact subset of the reals, 
there exists a primitive-like set connected with its 
Cantor-Bendixson derivative.

We recall that if $F$ is a closed subset of $\mathbb{R}$, then 
$(F^{(\alpha)})_{\alpha\in\mathbf{OR}}$ is a decreasing family of closed 
subsets of the real line. Furthermore,  if $K\in\mathcal{K}$, then 
$(K^{(\alpha)})_{\alpha\in\mathbf{OR}}$ is a decreasing family of 
elements of $\mathcal{K}$. 

We denote by $\mathbf{OR}$, the class of all ordinal numbers.  Moreover, 
$\omega$ is used to designate the set of all natural numbers and $\Omega$ 
represents the set of all countable ordinal numbers.  In addition, the 
cardinality of a set $B$ is denoted by $|B|$. 

\section{A family of elements in $\mathcal{K}$ having a Cantor-Bendixson's 
derivative with any given finite number of elements}\label{Sectfam} 

First, we remark that any finite subset of $\mathbb{R}$ is an element 
of $\mathcal{K}$ with empty derived  set.  Thus, a set of 
this kind satisfies the property that its Cantor-Bendixson's derivative is 
empty for all ordinal number greater than or equal to 1.  The following theorem 
let us find some elements belonging to $\mathcal{K}$ not satisfying this  
last property.  The main idea of the next result was given in \cite{Cantor1880}, 
for completeness, we present below its proof in detail.  

\begin{Theorem}\label{Th1}
For any countable ordinal number $\alpha \in \Omega$, and for all 
$a,b\in\mathbb{R}$ such that $a<b$, there is a set $K \in \mathcal{K}$ 
such that $K \subset (a,b]$ and $K^{(\alpha)}=\{b\}$.
\end{Theorem}
\begin{proof}
We will use Transfinite Induction. 
\begin{enumerate}[leftmargin=*]
\item[(a)]
First, we consider the case $\alpha=0$.  For any $a,b\in\mathbb{R}$ such that 
$a<b$, the result follows by taking the set $K=\{b\} \in \mathcal{K}$.  
\item[(b)]
Now, we suppose that for a given countable ordinal number $\alpha \in \Omega$, 
and for all  $c,d\in\mathbb{R}$ such that $c<d$, there is a set 
$\widetilde{K}\in \mathcal{K}$ such that $\widetilde{K}\subset (c,d]$ 
and $\widetilde{K}^{(\alpha)}=\{d\}$.  Let $a,b\in\mathbb{R}$ be such 
that $a<b$. We take a strictly increasing sequence, $(x_n)_{n\in\omega}$, in 
$(a,b]$ such that $x_n \rightarrow b$ as $n \rightarrow +\infty$. Defining 
$x_{-1}:=a$ and applying the hypothesis to the real numbers $x_{m-1} < x_m$,  
$m\in\omega$, it follows that there exists a sequence of sets 
$(K_m)_{m\in\omega}$ such that for all $m\in\omega$, $K_m\in \mathcal{K}$, 
$K_m\subset (x_{m-1},x_m]$ and $K_m^{(\alpha)}=\{x_m\}$.  Now, we define the set
  \begin{equation}\label{eq:02}
    K:=\biguplus_{m\in\omega} K_m \uplus\{b\}.
  \end{equation}
The set $K$, given in \eqref{eq:02}, satisfies the following properties:
\begin{itemize}
\item $K\subset (a,b]$, since $K_m\subset (x_{m-1},x_m] \subset (a,b]$, for all 
$m\in\omega$.
\item $K$ is countable, since it is the countable union of countable sets.
\item $K$ is compact.  In fact, given $(A_i)_{i\in I}$ an open cover of $K$, 
there is a $j\in I$ such that $b\in A_j$. Since $A_j$ is an open set and 
$(x_n)_{n\in\omega}$ is a strictly increasing sequence that converges to $b$, 
there exists $N_1\in\omega$ such that $K_n\subset A_j$ for all $n \in \omega$ with 
$n>N_1$. On the other hand, the set $C := \biguplus_{n=0}^{N_1} K_n$ is compact, 
since it is the finite union of compact sets.  Thus, $C$ has a finite open 
subcover $(A_i)_{i\in J}$. Then, $(A_i)_{i\in J \cup \{j\}}$ is a finite open 
subcover of $K$.
\item For all ordinal number $\beta$ with $\beta\leq\alpha$, 
  \begin{equation} \label{eq:03}
        K^{(\beta)}=\biguplus_{m\in\omega}K_m^{(\beta)}\uplus\{b\}.
  \end{equation}
Last expression is obtained by using Transfinite Induction on $\beta$.  In fact, 
the case $\beta=0$ is immediate from \eqref{eq:02}. Now, we suppose that 
for a given ordinal number $\beta<\alpha$,~\eqref{eq:03} holds. Since $\beta+1\leq \alpha$, 
we have that $K_m^{(\alpha)} \subset K^{(\alpha)} \subset K^{(\beta+1)}$, 
for all $m\in\omega$.  Moreover, since $x_m\in K_m^{(\alpha)} \subset K^{(\beta+1)}$, 
for all $m\in\omega$, and $x_m \rightarrow b$ as $m \rightarrow +\infty$, we 
see that $b\in K^{(\beta+1)}$.  Therefore,
    \begin{equation} \label{eq:03a}
        \biguplus_{m\in\omega}K_m^{(\beta+1)}\uplus\{b\}
				\subset K^{(\beta+1)}. 
    \end{equation}
In order to prove the other inclusion, let $x\in K^{(\beta+1)}$. Using the 
induction hypothesis, we see that
    \begin{equation*}
        K^{(\beta+1)}\subset K^{(\beta)}= \biguplus_{m\in\omega}
				K_m^{(\beta)}\uplus\{b\}.
    \end{equation*}
Therefore, either $x=b$ or $x\in K_m^{(\beta)}$ for some $m\in\omega$. If $x = b$, then 
there is nothing else to prove.  If $x\neq b$, there exists $M\in\omega$ such that
    \begin{equation*}
        x\in K_M^{(\beta)}\subset K_M \subset (x_{M-1},x_M].
    \end{equation*}
We claim that $x\in K_M^{(\beta+1)}$.  To prove the last assertion, we suppose, by 
contradiction, that $x\notin K_M^{(\beta+1)}$.  Thus, $x$ is an isolated point of 
$K_M^{(\beta)}$.  However, we know that $\{x_M\} =  K_M^{(\alpha)}\subset K_M^{(\beta+1)}$. 
Then, $x\neq x_M$.  Thus, there exists $\epsilon>0$ such that 
$(x-\epsilon,x+\epsilon)\subset (x_{M-1},x_M)$ and 
    \begin{equation*}
        (x-\epsilon,x+\epsilon)\cap K_M^{(\beta)}=\{x\}.
    \end{equation*}
Moreover, since $(x-\epsilon,x+\epsilon)\subset (x_{M-1},x_M)$, we conclude that 
for all $m\in\omega \smallsetminus \{M\}$, 
    \begin{equation*}
        (x-\epsilon,x+\epsilon)\cap K_m^{(\beta)}=\varnothing.
    \end{equation*}
Hence,
    \begin{align*}
       \{x\}&=(x-\epsilon,x+\epsilon)\cap 
			 \left(\biguplus_{m\in\omega}K_m^{(\beta)}\uplus\{b\}\right)\\[3pt]
       &=(x-\epsilon,x+\epsilon)\cap K^{(\beta)},
    \end{align*}
where in the last equality we have used the assumption that~\eqref{eq:03} 
holds for $\beta$.  Even so, this last expression is a contradiction 
with the fact that $x\in K^{(\beta+1)}$. Then, $x\in K_M^{(\beta+1)}$. Thus,
    \begin{equation} \label{eq:03b}
        K^{(\beta+1)}\subset \biguplus_{m\in\omega}K_m^{(\beta+1)}
				\uplus\{b\}.
    \end{equation}
Using~\eqref{eq:03a} and~\eqref{eq:03b}, we get
    \begin{equation*}
        K^{(\beta+1)} = \biguplus_{m\in\omega}K_m^{(\beta+1)}
				\uplus\{b\}.
    \end{equation*}
Finally, let $\gamma\neq 0$ be a limit ordinal such that $\gamma\leq\alpha$ 
and suppose that 
    \begin{equation} \label{eq:03c}
        K^{(\delta)}= \biguplus_{m\in\omega}K_m^{(\delta)}
				\uplus\{b\},
    \end{equation}
for all ordinal number $\delta$ such that $\delta<\gamma$.  Following a 
similar procedure to the one performed above to obtain~\eqref{eq:03a}, 
we have that 
    \begin{equation} \label{eq:03d}
        \biguplus_{m\in\omega}K_m^{(\gamma)}
				\uplus\{b\}\subset K^{(\gamma)}.
    \end{equation}
To obtain the other inclusion, let $x\in K^{(\gamma)}$. Using the induction 
hypothesis~\eqref{eq:03c}, we see that
    \begin{equation*}
        K^{(\gamma)} := \bigcap_{\delta<\gamma}  K^{(\delta)} = 
				\bigcap_{\delta<\gamma}\left(\biguplus_{m\in\omega}K_m^{(\delta)}
				\uplus\{b\}\right).
    \end{equation*}
Then, either $x=b$ or for all ordinal number $\delta$ such that 
$\delta<\gamma$, there exists $m\in\omega$ such that 
$x\in K_m^{(\delta)}$. If $x=b$, then there is nothing left to prove.  If 
$x\neq b$, there  exists $M\in\omega$ such that 
$x\in K_M^{(0)}=K_M\subset (x_{M-1},x_M]$.  We claim now that for all ordinal 
number $\delta$ such that $\delta<\gamma$, $x \in K_M^{(\delta)}$. In fact, 
we suppose, by contradiction, that there is an ordinal number $\delta_0$ with 
$\delta_0<\gamma$ and such that $x \notin K_M^{(\delta_0)}$.  However, we know 
that there exists $m_0\in\omega$ with $m_0 \neq M$ such that 
$x\in K_{m_0}^{(\delta_0)} \subset K_{m_0} \subset (x_{m_0-1},x_{m_0}]$.  Since 
$m_0 \neq M$, we get $(x_{m_0-1},x_{m_0}] \cap (x_{M-1},x_M] = \varnothing$, which is 
a contradiction with the fact that $x \in (x_{m_0-1},x_{m_0}] \cap (x_{M-1},x_M]$.  
Therefore,   
    \begin{equation*}
        x\in\bigcap_{\delta<\gamma}K_M^{(\delta)}=:K_M^{(\gamma)}
				\subset \biguplus_{m\in\omega}K_m^{(\gamma)}.
    \end{equation*}
Then,
    \begin{equation} \label{eq:03e}
        K^{(\gamma)}\subset\biguplus_{m\in\omega}K_m^{(\gamma)}\uplus\{b\}.
    \end{equation}
By~\eqref{eq:03d} and~\eqref{eq:03e}, we have that 
    \begin{equation*}
        K^{(\gamma)}=\biguplus_{m\in\omega}K_m^{(\gamma)}\uplus\{b\}.
    \end{equation*}
Hence,~\eqref{eq:03} holds for all ordinal number $\beta$ such that 
$\beta\leq\alpha$. 
\end{itemize}
Applying now~\eqref{eq:03} to the ordinal number $\alpha$, and since  
$K_m^{(\alpha)}=\{x_m\}$, for all $m\in\omega$, we conclude that 
    \begin{align*}
        K^{(\alpha)}&= \biguplus_{m\in\omega}K_m^{(\alpha)}\uplus\{b\}\\
            &=\biguplus_{m\in\omega}\{x_m\}\uplus\{b\}\\
            &=\{x_m:m\in\omega\}\uplus\{b\}.
    \end{align*}
Therefore,
    \begin{equation*}
        K^{(\alpha+1)}=(K^{(\alpha)})'=\{b\}.
    \end{equation*}	
\item[(c)]
Finally, let $\lambda\neq 0$ be a countable limit ordinal number.  We suppose 
that for all ordinal number $\rho$ such that $\rho<\lambda$ and for 
all $c,d\in\mathbb{R}$ such that $c<d$, there is a set 
$\widetilde{K}\in \mathcal{K}$ such that $\widetilde{K}\subset (c,d]$ 
and $\widetilde{K}^{(\rho)}=\{d\}$.  Since $\lambda$ is a countable limit 
ordinal number, there exits a strictly increasing sequence 
$(\rho_n)_{n\in\omega}$ in $\Omega$ such that $\rho_n<\lambda$, 
for all $n\in\omega$, and $\sup\{\rho_n:n\in \omega\}=\lambda$.  Let 
$a,b\in\mathbb{R}$ be such that $a<b$. We take a strictly increasing 
sequence, $(x_n)_{n\in\omega}$, in $(a,b]$ such that $x_n \rightarrow b$ 
as $n \rightarrow +\infty$. Defining again $x_{-1}=a$ and applying the 
hypothesis to the real numbers $x_{m-1} < x_m$,  and the ordinal 
number $\rho_m$, $m\in\omega$, it follows that there exists a sequence 
of sets $(K_m)_{m\in\omega}$ such that for all $m\in\omega$, 
$K_m\in \mathcal{K}$, $K_m\subset (x_{m-1},x_m]$ and 
$K_m^{(\rho_m)}=\{x_m\}$.  We also define, as in the previous case, the set
  \begin{equation} \label{eq:04}
    K:=\biguplus_{m\in\omega} K_m \uplus\{b\}.
  \end{equation}
It can be shown, similarly to the case (b) above, that the set $K$, 
defined in~\eqref{eq:04}, satisfies the following properties:
\begin{itemize}
  \item $K\subset (a,b]$.
  \item $K$ is countable.
  \item $K$ is compact.
  \item For all ordinal number $\rho$ with $\rho\leq\lambda$, 
    \begin{equation} \label{eq:04a}
          K^{(\rho)}=\biguplus_{m\in\omega}K_m^{(\rho)}\uplus\{b\}.
    \end{equation}
Last expression is obtained by using Transfinite Induction on $\rho$.  
In fact, the case $\rho=0$ is immediate from \eqref{eq:04}. Now, we suppose 
that for a given ordinal number $\rho<\lambda$,~\eqref{eq:04a} holds. 
Since $\lambda$ is a limit ordinal, we have that $\rho+1<\lambda$, and then 
there exists $N\in\omega$ such that $\rho+1<\rho_m$ for all $m \in \omega$ 
with $m>N$. Therefore, 
$x_m\in K_m^{(\rho_m)} \subset K_m^{(\rho+1)} \subset K^{(\rho+1)}$, 
for all $m \in \omega$ with $m>N$, and since $x_m \rightarrow b$ as 
$m \rightarrow +\infty$, we see that $b\in K^{(\rho+1)}$. Then,
    \begin{equation} \label{eq:04b}
        \biguplus_{m\in\omega}K_m^{(\rho+1)}\uplus\{b\}\subset K^{(\rho+1)}. 
    \end{equation}
In order to prove the other inclusion, let $x\in K^{(\rho+1)}$. Using 
the induction hypothesis, we see that
    \begin{equation*}
        K^{(\rho+1)}\subset K^{(\rho)}= 
				\biguplus_{m\in\omega}K_m^{(\rho)}\uplus\{b\}.
    \end{equation*}
Therefore, either $x=b$ or $x\in K_m^{(\rho)}$ for some $m\in\omega$. 
If $x = b$, then there is nothing else to prove.  If $x\neq b$, there 
exists $M\in\omega$ such that
    \begin{equation*}
        x\in K_M^{(\rho)}\subset K_M \subset (x_{M-1},x_M].
    \end{equation*}
Since $K_M^{(\rho_M+1)}=\varnothing$, we have that $\rho< \rho_M+1$, 
that is $\rho\leq\rho_M$. We claim that $x\in K_M^{(\rho+1)}$. To prove 
the last assertion, we suppose, by contradiction, that $x\notin K_M^{(\rho+1)}$.  
Thus, $x$ is an isolated point of $K_M^{(\rho)}$. However, we know 
that $K_M\cap K_{M+1}=\varnothing$, then $x\notin K_{M+1}$.  Hence, 
$x\notin K_{M+1}^{(\rho)}$. Thus, there exists $\epsilon>0$ such that 
$(x-\epsilon,x+\epsilon)\subset (x_{M-1},x_{M+1})$, 
$(x-\epsilon,x+\epsilon)\cap K_{M+1}^{(\rho)}=\varnothing$ and 
    \begin{equation*}
        (x-\epsilon,x+\epsilon)\cap K_M^{(\rho)}=\{x\},
    \end{equation*}
where in the second expression above we have used the fact that 
$K_{M+1}^{(\rho)}$ is a closed subset of $\mathbb{R}$.	 Moreover, 
since $(x-\epsilon,x+\epsilon)\subset (x_{M-1},x_{M+1})$, we 
conclude that for all $m\in\omega \smallsetminus \{M\}$, 
    \begin{equation*}
        (x-\epsilon,x+\epsilon)\cap K_m^{(\rho)}=\varnothing.
    \end{equation*}
Hence,
    \begin{align*}
    \{x\}&=(x-\epsilon,x+\epsilon)\cap 
		\left(\biguplus_{m\in\omega}K_m^{(\rho)}\uplus\{b\}\right)\\[3pt]
    &=(x-\epsilon,x+\epsilon)\cap K^{(\rho)},
    \end{align*}
where in the last equality we have used the assumption 
that~\eqref{eq:04a} holds for $\rho$. Nevertheless, this last 
expression is a contradiction with the fact that $x\in K^{(\rho+1)}$. 
Then, $x\in K_M^{(\rho+1)}$.  Thus,
    \begin{equation} \label{eq:04c}
        K^{(\rho+1)}\subset \biguplus_{m\in\omega}K_m^{(\rho+1)}
				\uplus\{b\}.
    \end{equation}
Using~\eqref{eq:04b} and~\eqref{eq:04c}, we get
    \begin{equation*}
        K^{(\rho+1)} = \biguplus_{m\in\omega}K_m^{(\rho+1)}
				\uplus\{b\}.
    \end{equation*}
Finally, let $\gamma\neq 0$ be a limit ordinal such that 
$\gamma\leq\lambda$ and suppose that 
    \begin{equation} \label{eq:04d}
        K^{(\delta)}= \biguplus_{m\in\omega}K_m^{(\delta)}
				\uplus\{b\},
    \end{equation}
for all ordinal number $\delta$ such that $\delta<\gamma$. We have, 
using~\eqref{eq:04d}, that 
    \begin{align} 
        \biguplus_{m\in\omega}K_m^{(\gamma)}\uplus\{b\}
        &=  \biguplus_{m\in\omega}\left(\bigcap_{\delta<\gamma}
				    K_m^{(\delta)}\right)\uplus\{b\}\notag\\
        &\subset  \bigcap_{\delta<\gamma}\left(\biguplus_{m\in\omega}
				    K_m^{(\delta)}\right)\uplus\{b\}\notag\\
        &=  \bigcap_{\delta<\gamma}\left(\biguplus_{m\in\omega}
				    K_m^{(\delta)}\uplus\{b\}\right)\notag\\
        &=\bigcap_{\delta<\gamma}K^{(\delta)}\notag\\
        &= K^{(\gamma)}.\label{eq:04e}
    \end{align}    
To get the other inclusion, we can follow a similar procedure to the one 
performed above to obtain~\eqref{eq:03e}.  Thus, we have that
	\begin{equation}\label{eq:04f}
	   K^{(\gamma)}\subset \biguplus_{m\in\omega}K_m^{(\gamma)}
	   \uplus\{b\}.
	\end{equation}
By~\eqref{eq:04e} and~\eqref{eq:04f}, we obtain 
    \begin{equation*}
        K^{(\gamma)}=\biguplus_{m\in\omega}K_m^{(\gamma)}
				\uplus\{b\}.
    \end{equation*}
Consequently,~\eqref{eq:04a} holds for all ordinal number $\rho$ 
such that $\rho\leq\lambda$. 
\end{itemize}
Furthermore, since for all $m\in\omega$, $\rho_m+1<\lambda$, it follows 
that for all $m\in\omega$
    \begin{equation*}
        K_m^{(\lambda)}\subset K_m^{(\rho_m+1)}
				=(K_m^{(\rho_m)})'=(\{x_m\})'=\varnothing.
    \end{equation*}
Therefore,
    \begin{equation*}
        K^{(\lambda)}=\biguplus_{m\in\omega}K_m^{(\lambda)}
				\uplus\{b\}=\{b\}.
    \end{equation*}
\end{enumerate}
From (a), (b) and (c), the theorem is proved.
\end{proof}
The next lemma will be used in the proof of Corollary~\ref{Cor1} 
below.
\begin{Lemma}\label{Lemtech0}
Suppose that $n\in \omega$. Let $F_1, F_2, \ldots, F_n$ be closed subsets 
of $\mathbb{R}$.  Then, for all ordinal number 
$\alpha \in \mathbf{OR}$, we have that 
\begin{equation*}
	\left(\bigcup_{k=1}^n F_k\right)^{(\alpha)} 
	= \bigcup_{k=1}^n F_k^{(\alpha)}.
\end{equation*}
\end{Lemma}
\begin{proof}
The general case, $n\in \omega$, is a consequence of the result  
for $n=2$ and the Principle of Finite Induction.  Thus, 
we suppose that $n=2$. We will now use Transfinite Induction.
\begin{enumerate}[label=(\alph*),leftmargin=*]
\item
If $\alpha = 0$, then there is nothing else to prove.
\item
We now suppose that for a given ordinal number $\alpha\in\mathbf{OR}$,  
$(F_1\cup F_2)^{(\alpha)} = F_1^{(\alpha)}\cup F_2^{(\alpha)}$. 
Therefore,  
\begin{equation*}
	(F_1\cup F_2)^{(\alpha+1)} = \big((F_1\cup F_2)^{(\alpha)}\big)'
	=\big(F_1^{(\alpha)}\cup F_2^{(\alpha)}\big)' 
	=F_1^{(\alpha+1)}\cup F_2^{(\alpha+1)},
\end{equation*}
where in the last equation we have used the fact that the derived set of a 
finite union of subsets of a metric space equals the union of their derived sets. 
\item 
Finally, let $\lambda\neq 0$ be a limit ordinal number. We suppose that for all 
$\beta\in\mathbf{OR}$ such that $\beta<\lambda$, 
$(F_1\cup F_2)^{(\beta)} = F_1^{(\beta)}\cup F_2^{(\beta)}$. Then, 
\begin{align*}
  F_1^{(\lambda)}\cup F_2^{(\lambda)} 
	  & = \bigcap_{\beta<\lambda} F_1^{(\beta)} \cup \bigcap_{\beta<\lambda} F_2^{(\beta)}\\
	  & \subset \bigcap_{\beta<\lambda} (F_1^{(\beta)} \cup  F_2^{(\beta)})\\
	  & = \bigcap_{\beta<\lambda} (F_1 \cup  F_2)^{(\beta)}\\
	  & = (F_1 \cup  F_2)^{(\lambda)}.
\end{align*}
In order to prove the other inclusion, we take $x\in (F_1 \cup  F_2)^{(\lambda)}$.
We suppose, for the sake of contradiction, that $x\not\in F_1^{(\lambda)}$ and 
$x\not\in F_2^{(\lambda)}$.  Thus, there exist $\beta_1, \beta_2 \in \mathbf{OR}$,
with $\beta_1<\lambda$ and $\beta_2<\lambda$, such that $x\not\in F_1^{(\beta_1)}$ and 
$x\not\in F_2^{(\beta_2)}$. If $\beta_1\leq \beta_2$, then  
$F_1^{(\beta_2)}\subset F_1^{(\beta_1)}$.  Hence, 
$x\not\in F_1^{(\beta_2)}\cup F_2^{(\beta_2)}=(F_1\cup F_2)^{(\beta_2)}$,
which contradicts the fact that 
$x\in (F_1 \cup  F_2)^{(\lambda)}=\bigcap_{\beta<\lambda} (F_1 \cup  F_2)^{(\beta)}$. 
The proof of the other case, $\beta_2 < \beta_1$, is similar.  Therefore,
\begin{equation*}
	(F_1 \cup  F_2)^{(\lambda)} = F_1^{(\lambda)}\cup F_2^{(\lambda)}.
\end{equation*}
\end{enumerate}
Consequently, the lemma is proved.
\end{proof}
The following result is a generalization of Theorem~\ref{Th1}.
\begin{Corollary}\label{Cor1}
Given any countable ordinal number $\alpha$ and given any $p\in\omega$, 
there exists $K\in \mathcal{K}$ such that $|K^{(\alpha)}|=p$.
\end{Corollary}
\begin{proof}
Let $\alpha \in \Omega$. If $p=0$, we take $K=\varnothing$.  
If $p\in\omega \smallsetminus\{0\}$, it is enough to apply 
Theorem~\ref{Th1} to a collection of $p$ pairwise disjoint intervals.  Thus, 
for all $k \in \{1,\ldots, p\}$, there exists $K_k\in\mathcal{K}$, such that 
$K_k^{(\alpha)}$ has only one element, and $K_i \cap K_j = \varnothing$ for 
$i, j \in \{1, \ldots, p\}$ with $i \neq j$. We now define
\begin{equation*}
	K:=\biguplus_{k=1}^p K_k.  
\end{equation*}
Hence, $K\in \mathcal{K}$ and, using Lemma~\ref{Lemtech0}, we get
\begin{equation*}
	K^{(\alpha)}=\biguplus_{k=1}^p K_k^{(\alpha)}.
\end{equation*}
Therefore, $K^{(\alpha)}$ has exactly $p$ elements.
\end{proof}
\begin{Remark}
Even though the proofs of~\eqref{eq:03} and~\eqref{eq:04a} are similar, it is worth 
mentioning that they are not identical.  In fact, to prove~\eqref{eq:03} we have that  
$\alpha \in \Omega$ and for all $m \in \omega$, $K_m^{(\alpha)}=\{x_m\}$.  On 
the other hand, to obtain~\eqref{eq:04a} we consider $\lambda \neq 0$ a 
countable limit ordinal and a strictly increasing sequence 
$(\rho_m)_{m\in\omega}$ in $\Omega$, with $\sup\{\rho_m:m\in \omega\}=\lambda$, 
such that for all $m \in \omega$, $\rho_m<\lambda$ and $K_m^{(\rho_m)}=\{x_m\}$, where 
$\rho_m$ depends on $m$. \\
In addition, we point out that the process developed to obtain~\eqref{eq:04e} 
can also be used to get~\eqref{eq:03a}, \eqref{eq:03d} and~\eqref{eq:04b}.
\end{Remark}

\section{Some results concerning Cantor-Bendixson's derivative}

It is a well-known fact that, for all $K\in\mathcal{K}$, 
$(K^{(\alpha)})_{\alpha\in\mathbf{OR}}$ is a decreasing family of 
elements of $\mathcal{K}$.  The following two results were first 
proved by G. Cantor in~\cite{Cantor1883s} and they 
imply that for all $K\in\mathcal{K}$, 
$(K^{(\alpha)})_{\alpha\in\mathbf{OR}}$ is in fact a strictly 
decreasing family of sets in $\mathcal{K}$ up to a countable 
ordinal number and such that all of its subsequent derivative 
sets are empty.
\begin{Lemma}
If $K\in\mathcal{K}$ and $K\neq\varnothing$, then $K'\neq K$.
\end{Lemma}
The above lemma implies the following theorem. 
\begin{Theorem}
If $K\in \mathcal{K}$, then there exists a countable ordinal 
number $\beta$ such that $K^{(\beta)}$ is finite.
\end{Theorem}
Since $\Omega$ is a well-ordered set, by the previous theorem, we see that 
for all $K\in\mathcal{K}$, there exists the smallest countable ordinal 
number $\alpha$ such that $K^{(\alpha)}$ is finite.   We can now 
give the next definition.
\begin{Definition}[Cantor-Bendixson's characteristic]
Let $K\in \mathcal{K}$.  We say that $(\alpha,p)\in \Omega\times\omega$ 
is the \emph{Cantor-Bendixson characteristic} of $K$ if $\alpha$ is 
the smallest countable ordinal number such that $K^{(\alpha)}$ is 
finite and $|K^{(\alpha)}|=p$. In this case, we write 
$\mathcal{CB}(K)=(\alpha,p)$.
\end{Definition}
By Theorem~\ref{Th1}, for all countable ordinal number $\alpha$, there 
exists a set $K \in \mathcal{K}$ having Cantor-Bendixson's characteristic 
$(\alpha,1)$. Furthermore, by Corollary~\ref{Cor1}, we have that
for all $p\in\omega\smallsetminus\{0\}$ and for all $\alpha \in \Omega$, 
there exists $K\in \mathcal{K}$ such that $\mathcal{CB}(K)=(\alpha,p)$. 
In addition, we obviously see that $\mathcal{CB}(\varnothing)=(0,0)$.
Moreover, we have the next result concerning the empty set.
\begin{Proposition}\label{Prop1}
Let $K\in\mathcal{K}$ be such that 
$\mathcal{CB}(K)=(\alpha,p) \in \Omega \times \omega$. Then, $p= 0$ 
if and only if $K=\varnothing$.
\end{Proposition}
\begin{proof}
If $K=\varnothing$, then $\mathcal{CB}(K)=(0,0)$, and thus the result holds.
\noindent
Now, we suppose that $K\neq\varnothing$. We consider three cases.
\begin{itemize}
\item If $\alpha=0$, then $K=K^{(0)}$ is finite.  Since 
      $K\neq\varnothing$, we have that $|K^{(0)}|\neq 0$. Hence, $p\neq0$.
\item We suppose now that $\alpha$ is a nonzero limit ordinal.  Then, 
      for all  $\beta \in \Omega$ such that $\beta <\alpha$, $K^{(\beta)}$ 
			is infinite.  Therefore, $(K^{(\beta)})_{\beta<\alpha}$ is a decreasing 
			nested family of nonempty compact subsets of $\mathbb{R}$. By using 
			the Cantor Intersection Theorem, we obtain
       \begin{equation*}
           K^{(\alpha)}=\bigcap_{\beta<\alpha}K^{(\beta)}\neq\varnothing.
       \end{equation*}
      Then, $|K^{(\alpha)}|\neq 0$, and so $p\neq0$.
\item Finally, we assume that $\alpha$ is a successor ordinal.  Thus, there 
      exists an ordinal $\beta \in \Omega$  such that $\beta+1=\alpha$. Since 
			$\beta<\alpha$, it follows that $K^{(\beta)}$ is infinite.  Then, 
       \begin{equation*}
           K^{(\alpha)}=K^{(\beta+1)}=(K^{(\beta)})'\neq\varnothing.
       \end{equation*}
    Therefore, $|K^{(\alpha)}|\neq 0$.  Hence, $p\neq0$.\qedhere
\end{itemize}
\end{proof}
\subsection{Partition of $\mathcal{K}$}
In this subsection, we show some general results concerning the 
equivalence relation $\sim$ defined on the set $\mathcal{K}$ 
by \eqref{relation1}.  
\begin{Proposition}
Let $K_1, K_2\in \mathcal{K}$ be such that $K_1\sim K_2$. Then, 
$K'_1\sim K'_2$. More precisely, if $f$ is a homeomorphism 
of $K_1$ onto $K_2$, then $f|_{K'_1}$ is also a homeomorphism 
of $K_1'$ onto $K_2'$.
\end{Proposition}
\begin{proof}
Since the image of a limit point, under a homeomorphism, is also a limit 
point, we see that $f(K'_1)=K'_2$. Hence, 
$f|_{K'_1}\colon K'_1 \longmapsto K'_2$ is a homeomorphism. Therefore, 
$K'_1\sim K'_2$.
\end{proof}
By using Transfinite Induction, we get the following result.
\begin{Corollary}
Let $K_1,K_2\in \mathcal{K}$ be such that $K_1\sim K_2$, and 
let $\alpha$ be any ordinal number.  Then, 
$K^{(\alpha)}_1\sim K^{(\alpha)}_2$.  More precisely, if $f$ is 
a homeomorphism of $K_1$ onto $K_2$, then $f|_{K^{(\alpha)}_1}$ 
is also a homeomorphism of $K^{(\alpha)}_1$ onto $K^{(\alpha)}_2$.
\end{Corollary}
It follows from the last corollary that if $K_1,K_2\in \mathcal{K}$, 
$K_1\sim K_2$ and $\mathcal{CB}(K_1)=(\alpha,p) \in \Omega \times \omega$, 
then there exists a bijective function of $K^{(\alpha)}_1$ onto 
$K^{(\alpha)}_2$.  Therefore, $|K_2^{(\alpha)}|=|K_1^{(\alpha)}|=p$. 
Hence,  $\mathcal{CB}(K_2)=(\alpha,p)$. This last result about the 
Cantor-Bendixson characteristic, which was given by 
S. Mazurkiewicz and W. Sierpinski in~\cite{Sierpinski}, 
is expressed in the following theorem.
\begin{Theorem}\label{Th2}
If $K_1,K_2\in \mathcal{K}$ and $K_1\sim K_2$, then 
$\mathcal{CB}(K_1)=\mathcal{CB}(K_2)$.
\end{Theorem}
The above theorem shows that the Cantor-Bendixson characteristic is 
preserved for equivalent elements of $\mathcal{K}$, i.e., given 
$K\in\mathcal{K}$, we have that $\mathcal{CB}(K_1)=\mathcal{CB}(K)$, 
for all $K_1\in [K]$, where $[K]$ denotes the equivalence class of $K$. 
The reciprocal of Theorem \ref{Th2}, which was likewise given by 
S. Mazurkiewicz and W. Sierpinski in~\cite{Sierpinski}, is also true, 
and for completeness we give a more explicit proof of this fact 
in Theorem \ref{Th3} below. In the following, we consider any ordinal 
number as a topological space with the order topology.  Lemmas \ref{Lem2} 
to \ref{Lem5} will be used in the proof of Theorem \ref{Th3}.
\begin{Lemma}\label{Lem2}
Let $K\in\mathcal{K}$ be such that $\mathcal{CB}(K)=(1,1)$. Then, 
there exists a homeomorphism of $K$ onto $\omega+1$. 
\end{Lemma}
\begin{proof}
There is an $x\in \mathbb{R}$ such that  $K'=\{x\}$. The set 
$K\smallsetminus K'$ is infinite and countable. Therefore, there 
exists a bijective function $g$ of $K\smallsetminus K'$ onto $\omega$. 
Now, we define
\begin{equation*}
    \funcion{f}{K}{\omega+1}{z}{f(z)=\begin{cases}
    g(z),& \text{if }z\neq x,\\
    \omega,& \text{if }z= x.
    \end{cases}}
\end{equation*}
We see that $f$ is a bijective function.  Furthermore, since 
$\omega + 1$ is a compact topological space, $(\omega+1)' = \{\omega\}$, 
$f$ is an injective function, and $f(K') = f(\{x\}) = \{\omega\}$, 
we have that $f$ is a continuous function.  Moreover, 
since $\omega+1$ is a Hausdorff space, it follows that 
$f$ is in fact a homeomorphism.
\end{proof}
\begin{Lemma}\label{Lem3}
Let $\alpha$ be a countable ordinal number such that $\alpha>1$. Suppose 
that for all ordinal number $\beta$ such that $0<\beta<\alpha$ and for all 
$\widetilde{K}\in\mathcal{K}$ such that 
$\mathcal{CB}(\widetilde{K})=(\beta,p) 
\in \Omega \times (\omega \smallsetminus \{0\})$, there 
exists a homeomorphism $\widetilde{f}$ of $\widetilde{K}$ onto 
$\omega^\beta\cdot p+1$. Then, for all 
$K\in\mathcal{K}$ such that $\mathcal{CB}(K)=(\alpha,1)$, there 
exists a homeomorphism of $K$ onto $\omega^\alpha+1$.
\end{Lemma}
\begin{proof}
Let $K\in\mathcal{K}$ be such that $\mathcal{CB}(K)=(\alpha,1)$.  
Then, there exists an $x \in K$ such that  $K^{(\alpha)}=\{x\}$. We have 
that $x\in K^{(\alpha)}\subset K''$. Thus, $x$ is a limit point of $K'$. 
Hence, there exists a strictly increasing or strictly decreasing 
sequence $(x_n)_{n\in\omega}$ in $K' $ such that it converges to $x$.   
We suppose that $(x_n)_{n\in\omega}$ is an strictly increasing sequence 
in $K'$, the other case is similar. \\
We claim that for all $n\in\omega$, we can take $ r_n>0$ such that 
$x_n<x- r_n<x_{n+1}$ and $x- r_n, x+ r_n \notin K$. 
In fact, if we suppose the contrary, then there exists $l \in \omega$  
such that 
\begin{equation*}
   [x_l-x,x_{l+1}-x]\subset \{ r\in \mathbb{R}: 
   x- r\in K\text{ or }x+ r\in K\}.
\end{equation*}
However, the set on the right-hand side of the last inclusion is 
countable, which is a contradiction.  Hence, the claim is proved.  
We remark that the sequence $( r_n)_{n\in\omega}$ 
converges to 0 as $n$ goes to infinity.  We now define the sets
\begin{align}\label{eq:sets}
  K_0  &=K\cap \big((-\infty,x- r_0]\cup[x+ r_0,+\infty)\big),
	\nonumber \\
  K_{k}&=K\cap \big([x- r_{k-1},x- r_{k}]
	      \cup [x+ r_{k},x+ r_{k-1}]\big), 
				\; \; k \in \omega \smallsetminus \{0\}.
\end{align}
We see that for all $k \in \omega$, $x_k \in K_k$.  In addition, 
the sequence of sets $(K_k)_{k\in\omega}$ satisfies the 
following properties.
\begin{itemize}
 \item $K_k\subset K$, for all $k\in\omega$.
 \item $K_k\in \mathcal{K}$, for all $k\in\omega$, since they are countable 
       closed subsets of $K$.
 \item $x_k \in K_k'\neq\varnothing$, for all $k\in\omega$.  In fact, 
       let $\varepsilon >0$.  First, we consider the case 
			 $k\in\omega \smallsetminus \{0\}$.  We now take 
			 $\hat\varepsilon := \min \{\varepsilon, x_k-x +  r_{k-1}, 
			  x- r_k-x_k\}>0$.  Since $x_k \in K'$, there exists 
			 $z \in [(x_k-\hat\varepsilon, x_k+\hat\varepsilon) \smallsetminus \{x_k\}] \cap K$.  
			 Thus, $z \in [(x_k-\varepsilon, x_k+\varepsilon) \smallsetminus \{x_k\}] \cap K_k$.  
			 Hence, $x_k \in K_k'$.  For the case $k=0$, by taking  
			 $\hat\varepsilon := \min \{\varepsilon, x- r_0-x_0\}>0$, and proceeding in a 
			 similar way as in the previous case, we see that $x_0 \in K_0'$. 
 \item $(K_k)_{k\in\omega}$ is a pairwise disjoint sequence in $\mathcal{K}$.
 \item $\displaystyle \biguplus_{k\in\omega} K_k\uplus\{x\}=K$.  The fact that 
       $\displaystyle \biguplus_{k\in\omega} K_k\uplus\{x\} \subset K$ follows directly 
			 from \eqref{eq:sets}.  In order to prove the reverse inclusion, we take $z \in K$.  
			 If $z = x$, there is nothing else to show.  Now, we suppose that $z\neq x$.  Since 
			 $ r_n \to 0$ as $n \to +\infty$, we can choose the smallest natural number 
			 $N \in \omega$ such that $ r_N < |x-z|$.  Then, $z \in K_N$.
\end{itemize}
Moreover, from \eqref{eq:sets} we see that for all $k\in\omega$, 
$x\not\in K_k^{(\alpha)}\subset\{x\}$. Therefore, for all $k \in \omega$, 
$K_k^{(\alpha)}=\varnothing$. 
Thus, for all $k \in \omega$, 
$\mathcal{CB}(K_k)=(\beta_k,p_k) \in \Omega \times \omega$ implies 
that $0<\beta_k<\alpha$.  We remark that for all $k \in \omega$, 
$K_k \neq \varnothing$ implies that 
$p_k \in \omega \smallsetminus \{ 0\}$.  
 Using the hypothesis, we conclude that for all 
$k\in\omega$, there exists a homeomorphism $f_k$ of $K_k$ onto 
$\omega^{\beta_k}\cdot p_k+1$. We now define the function
\begin{equation*}
		\funcion{f}{K}{\tau+1}{z}
		{f(z)=\begin{cases}
		   f_0(z), & \text{if }z\in K_0,\\[2mm]
			 \displaystyle
		   \sum_{j=0}^{k-1}\omega^{\beta_j}\cdot p_j+1+f_k(z), 
			 & \text{if }z\in K_k, \; 
			 k \in \omega\smallsetminus\{0\}, \\[2mm]
			 \tau,  & \text{if }z=x,
		\end{cases}}
\end{equation*}
where
\begin{equation*}
  \tau:=\sum_{k\in\omega}\omega^{\beta_k}\cdot p_k
	:=\sup\left\{\sum_{k=0}^n\omega^{\beta_k}\cdot p_k:n  
	\in\omega\right\}.
\end{equation*}
\begin{enumerate}[label=(\alph*),leftmargin=*]
\item
First, we remark that $f$ is an injective function.  In fact, 
let $u, v \in K$ be such that $f(u)=f(v)$.  If $u=x$ and 
$v \in K_q$, for some $q \in \omega$, then 
$f(v) \le \sum_{k=0}^{q} \omega^{\beta_k} \cdot p_k < \tau =f(u)$, 
which is a contradiction.  Thus, there exists $r \in \omega$ 
such that $u \in K_r$.  We suppose, by contradiction, that 
$q \neq r$.  Without loss of generality, we may assume that $q < r$.  
Then, 
\begin{align*}
  f(v) &\le \sum_{k=0}^q \omega^{\beta_k} \cdot p_k
        \le \sum_{k=0}^{r-1} \omega^{\beta_k} \cdot p_k \\
			 &< \sum_{k=0}^{r-1} \omega^{\beta_k} \cdot p_k +1 
			    + f_r(u) 
			  = f(u),
\end{align*}
which is not possible.  Hence, $q=r$.  Thus, 
\begin{equation*}
  \sum_{k=0}^{q-1} \omega^{\beta_k} \cdot p_k +1+f_q(u) = f(u) 
	= f(v) = \sum_{k=0}^{q-1} \omega^{\beta_k} \cdot p_k +1+f_q(v),
\end{equation*}
implies that $f_q(u) = f_q(v)$.  Using the fact that $f_q$ is 
an injective function, it follows that $u=v$.
\item
We will now show that $f$ is onto.  In fact, let $\gamma \le \tau$.  
If $\gamma = \tau$, we have that $f(x) = \tau =\gamma$.  If 
$\gamma < \tau$, we take 
$M:= \min \{n \in \omega : \gamma \le 
\sum_{k=0}^n \omega^{\beta_k} \cdot p_k\}$.  
In case $M=0$, $\gamma \le \omega^{\beta_0} \cdot p_0$.  Since, $f_0$ 
is onto, there exists $z \in K_0 \subset K$ such that $f(z)= f_0(z)= \gamma$.  
We now assume that $M \in \omega \smallsetminus \{ 0 \}$.  Then, 
\begin{equation*}
 \sum_{k=0}^{M-1} \omega^{\beta_k} \cdot p_k +1 \le \gamma
 \le \sum_{k=0}^{M} \omega^{\beta_k} \cdot p_k.
\end{equation*}
Thus, there exists an ordinal number $\mu$ such that
\begin{equation*}
 \sum_{k=0}^{M-1} \omega^{\beta_k} \cdot p_k +1 + \mu = \gamma
 \le \sum_{k=0}^{M-1} \omega^{\beta_k} \cdot p_k 
 +\omega^{\beta_M} \cdot p_M.
\end{equation*} 
Then, $\mu \le \omega^{\beta_M} \cdot p_M$.  Since $f_M$ is onto, 
there exists $z \in K_M \subset K$ such that $f_M(z) = \mu$.  So, 
$f(z)= \sum_{k=0}^{M-1} \omega^{\beta_k} \cdot p_k +1+f_{M}(z)
=\gamma$. 
\item
Moreover, for all $k \in \omega$, $f|_{K_k}$ equals an ordinal number, i.e. 
a constant function,  plus a continuous function.  Thus, for 
all $k \in \omega$, $f|_{K_k}$ is a continuous function. In addition, 
since $(K_k)_{k\in\omega}$ is a pairwise disjoint sequence of open  
subsets in $K$, it follows that $f$ is a continuous function
at any element of $\biguplus_{k\in\omega} K_k$. Furthermore, $f$ is also 
continuous at the point $x \in K$.  If fact, let $\mu$ be an ordinal 
number such that $\mu < \tau$.  There exists $m \in \omega$ such that 
$\mu < \sum_{j=0}^m \omega^{\beta_j} \cdot p_j$.  We claim that 
\begin{equation}\label{eq:continuityx}
  f((x- r_m, x+ r_m) \cap K) \subset (\mu, \tau +1).
\end{equation}
Let $y \in (x- r_m, x+ r_m) \cap K$.  If $y=x$, then 
$f(y)=f(x)=\tau \in (\mu, \tau +1)$.  We now suppose that 
$y \neq x$.  Then, there is $i \in \omega$ such that 
$y \in K_i$. Since $( r_n)_{n \in \omega}$ is a 
strictly decreasing sequence of positive numbers, we conclude 
that $i > m$.  Then, 
\begin{equation}\label{eq:cont1}
  f(y) = \sum_{j=0}^{i-1} \omega^{\beta_j} \cdot p_j + 1 + f_i(y) 
	     \ge \sum_{j=0}^{m} \omega^{\beta_j} \cdot p_j 
			 > \mu.
\end{equation} 
Moreover,
\begin{align}\label{eq:cont2}
  f(y) & = \sum_{j=0}^{i-1} \omega^{\beta_j} \cdot p_j + 1 + f_i(y) 
	       \le \sum_{j=0}^{i-1} \omega^{\beta_j} \cdot p_j + 1 
			     + \omega^{\beta_i} \cdot p_i  \nonumber \\
			 & = \sum_{j=0}^{i} \omega^{\beta_j} \cdot p_j \le \tau 
			   < \tau +1.
\end{align}
From \eqref{eq:cont1} and \eqref{eq:cont2}, we see that 
$f(y) \in (\mu, \tau+1)$.  Thus, \eqref{eq:continuityx} 
follows.  Hence, $f$ is continuous at the point $x$. 
\end{enumerate}
By (a) and (b), $f$ is a bijective function.  In addition, by (c),
$f$ is a continuous function of $K$ onto $\tau+1$.  \\
We will now prove that $\tau=\omega^\alpha$. In order to get  
this, let  
$\widetilde{\alpha}:=\sup\{\beta_k:k\in\omega\} \in \mathbf{OR}$. 
We see that $\widetilde{\alpha}\leq \alpha$. 
\begin{enumerate}[label= (\roman*)]
\item
First, we consider the case $\widetilde{\alpha}< \alpha$.  Then,  
$\widetilde{\alpha} +1 \le \alpha$.  Thus, for all $k \in \omega$, 
$K_k^{(\widetilde{\alpha}+1)}=\varnothing$.  Using 
Transfinite Induction, and proceeding as in the proof 
of \eqref{eq:03}, we get   
\begin{equation*}
  K^{(\widetilde{\alpha}+1)} 
	= \biguplus_{k \in \omega} K_k^{(\widetilde{\alpha}+1)}
	  \uplus \{x\} = \{x\}.
\end{equation*}
Then, $\widetilde{\alpha}+1 = \alpha$. Since for all 
$k \in \omega$, 
$\omega^{\beta_k} \cdot p_k \le \omega^{\widetilde{\alpha}} \cdot p_k$, 
we see that 
\begin{equation}\label{eq:tau1}
  \tau = \sum_{k \in \omega}\omega^{\beta_k} \cdot p_k 
  \le  \omega^{\widetilde{\alpha}}  \cdot
	\left(\sum_{k \in \omega} p_k\right)
  = \omega^{\widetilde{\alpha}} \cdot \omega 
  = \omega^{\widetilde{\alpha}+1} = \omega^{\alpha}.
\end{equation}	
On the other hand, we claim that  
\begin{equation}\label{eq:cardset}
   |\{ n \in \omega : \beta_n = \widetilde{\alpha} \}| 
	 = \aleph_0.
\end{equation}	
In order to prove \eqref{eq:cardset}, we first suppose, by 
contradiction, that for all $n \in \omega$, 
$\beta_n < \widetilde{\alpha}$.  Thus, for all $n \in \omega$, 
$\beta_n+1 \le \widetilde{\alpha}$, and we get 
$K_n^{(\widetilde{\alpha})} \subset K_n^{(\beta_n+1)} = \varnothing$.  
Moreover, we see that 
$K^{(\widetilde{\alpha})} = \biguplus_{k \in \omega} 
K_k^{(\widetilde{\alpha})} \uplus \{x\} = \{x\}$.  Then, 
$\widetilde{\alpha} = \alpha$, which is a contradiction.  Hence, 
there exists at least one $n\in \omega$ such that 
$\beta_n = \widetilde{\alpha}$.  We now suppose, again by contradiction, 
that the set 
$\{ n \in \omega : \beta_n = \widetilde{\alpha} \} \neq \varnothing$
is finite.  Let 
$N := \max \{n \in \omega : \beta_n = \widetilde{\alpha} \} \in \omega$. 
We have that for all $k \in \omega$ such that $k>N$, 
$\beta_k < \widetilde{\alpha}$.  Then,
\begin{equation*}
  K^{(\widetilde{\alpha})} = \biguplus_{k \in \omega} 
  K_k^{(\widetilde{\alpha})} \uplus \{x\} 
	= \biguplus_{k=0}^N 
  K_k^{(\widetilde{\alpha})} \uplus \{x\}. 
\end{equation*}
It follows that,  $K^{(\widetilde{\alpha})}$ is a finite set.  Hence, 
$K^{(\alpha)} = K^{(\widetilde{\alpha} +1)} = \varnothing$, 
which is a contradiction with the fact that 
$K^{(\alpha)} = \{x\}$.  Therefore, \eqref{eq:cardset} is proved.  
We now define, for all $n \in \omega$, 
\begin{equation*}
  m_n := |\{k \in \omega : k \le n \text{ and } 
	\beta_k = \widetilde{\alpha}  \}| \in \omega.
\end{equation*}
Then, for all $n \in \omega$, we have that
\begin{equation*}
  \sum_{k=0}^n \omega^{\beta_k} \cdot p_k 
	\ge \omega^{\widetilde{\alpha}} \cdot m_n.
\end{equation*}
For this reason,
\begin{align}\label{eq:tau2}
	\tau & = \sum_{k\in\omega}\omega^{\beta_k}\cdot p_k 
	       \ge \omega^{\widetilde{\alpha}} 
	           \cdot \sup \{m_n : n\in \omega \} \nonumber \\
			 & = \omega^{\widetilde{\alpha}} \cdot \omega
			   = \omega^{\widetilde{\alpha}+1}
         = \omega^\alpha.
\end{align}
Using \eqref{eq:tau1} and \eqref{eq:tau2}, we conclude that 
$\tau = \omega^\alpha$.
\item
We now consider the case $\widetilde{\alpha} = \alpha$. We claim that 
for all $k \in \omega$, $\beta_k < \widetilde{\alpha}$.  In fact, 
if there exists $l \in \omega$ such that $\beta_l = \widetilde{\alpha}$, 
then 
\begin{equation*}
  K_l^{(\beta_l)} \uplus \{x\} \subset 
	\biguplus_{i\in \omega} K_i^{(\beta_l)} \uplus \{x\} 
	= K^{(\beta_l)} = K^{(\alpha)} = \{x\},
\end{equation*}
contradicting the fact that $|K_l^{(\beta_l)}|=p_l>0$. We now 
remark that $\alpha$ is a limit ordinal.  In order to prove 
the last assertion, we suppose, for the sake of contradiction, 
that $\alpha$ is a successor ordinal. Then, there exists an ordinal 
number $\lambda$ such that $\alpha = \lambda +1$.  Thus, for all $k \in \omega$,
$\beta_k \le \lambda < \alpha= \widetilde{\alpha}$, which is a 
contradiction with the definition of $\widetilde{\alpha}$.  
On the other hand, since for all $k \in \omega$, 
$\omega^{\beta_k} \le \omega^{\beta_k} \cdot p_k \le \tau$, it 
follows that 
\begin{equation}\label{eq:tau3}
  \omega^{\alpha} = \omega^{\widetilde{\alpha}} 
	=\sup \{\omega^{\beta_k} : k \in \omega \} \le \tau.
\end{equation}
We now define, for all $n \in \omega$,
\begin{align*}
  \beta_{k_n} & := \max \{\beta_k : k=0,1,\ldots,n \}, \\
	p_{k_n}     & := \max \{ p_k : k=0,1,\ldots,n \}.
\end{align*}
Then, for all $n \in \omega$, we see that 
\begin{equation*}
  \sum_{k=0}^n \omega^{\beta_k} \cdot p_k 
	\le \omega^{\beta_{k_n}} \cdot p_{k_n} \cdot  n 
	< \omega^{\beta_{k_n}+1} \le \omega^\alpha,
\end{equation*}
where in the last inequality we have used the fact 
that $\beta_{k_n}<\beta_{k_n}+1 \le \alpha$.  In 
consequence,
\begin{equation}\label{eq:tau4}
   \tau = \sum_{k\in\omega}\omega^{\beta_k}\cdot p_k
	 \le \omega^\alpha.
\end{equation}
Equations \eqref{eq:tau3} and  \eqref{eq:tau4} imply that
$\tau = \omega^\alpha$.
\end{enumerate}
Therefore, $f$ is a bijective and continuous function 
of $K$ onto $\tau+1=\omega^{\alpha}+1$. 
In addition, since $\omega^{\alpha}+1$ is a Hausdorff space,
we conclude that $f$ is a homeomorphism of $K$ 
onto $\omega^\alpha+1$.
\end{proof}
\begin{Lemma}\label{Lemtech}
Suppose that $K$ and $F$ are closed subsets of $\mathbb{R}$ such 
that $K \cap F$ = $K \cap \accentset{\circ}{F}$, where 
$\accentset{\circ}{F}$ is the set of all interior points of $F$.  
Then, for all $\alpha \in \mathbf{OR}$, we have that 
\begin{equation}\label{eq:lemtech}
  (K \cap F)^{(\alpha)} = K^{(\alpha)} \cap F.
\end{equation} 
\end{Lemma}
\begin{proof}
We proceed by Transfinite Induction.  
\begin{itemize}[leftmargin=*]
\item
The case $\alpha = 0$ is immediate. 
\item
We now suppose that the result is true for $\alpha \in \mathbf{OR}$.  Then, 
\begin{equation*}
  (K \cap F)^{(\alpha +1)} = \left( (K \cap F)^{(\alpha)} \right)'
	                         = ( K^{(\alpha)} \cap F )'
													 \subset ( K^{(\alpha)} )' \cap F'
													 \subset K^{(\alpha +1)} \cap F,
\end{equation*}
where in the last expression we have used the induction hypothesis and 
the fact that $F$ is closed.  In order to prove the reverse inclusion, 
let $x \in K^{(\alpha +1)} \cap F$.  Since $K$ is closed, 
$x \in K \cap F = K \cap \accentset{\circ}{F}$.  Thus, there exists $ r>0$ such 
that $(x- r, x+ r) \subset F$.  Let $\varepsilon >0$.  We now take 
$\tilde{\varepsilon}:= \min \{ \varepsilon,  r\} >0$.  Then,
\begin{align*}
  \varnothing & \neq \big( (x-\tilde{\varepsilon}, x+\tilde{\varepsilon}) 
	              \smallsetminus \{ x\} \big) \cap K^{(\alpha)}
							 =\big( (x-\tilde{\varepsilon}, x+\tilde{\varepsilon}) 
							  \smallsetminus \{ x\} \big) \cap K^{(\alpha)} \cap F \\
							& \subset \big( (x- \varepsilon, x+\varepsilon) 
							  \smallsetminus \{ x\} \big) \cap (K \cap F)^{(\alpha)}.
 \end{align*}
Hence, $x \in (K \cap F)^{(\alpha +1)}$.  Therefore, 
$(K \cap F)^{(\alpha +1)} = K^{(\alpha +1)} \cap F$.
\item
Finally, let $\lambda \neq 0$ be a limit ordinal number.  We suppose that 
for all $\beta \in \mathbf{OR}$ such that $\beta < \lambda$, 
$(K \cap F)^{(\beta)} = K^{(\beta)} \cap F$.  Then, 
\begin{equation*}
  (K \cap F)^{(\lambda)} = \bigcap_{\beta < \lambda} (K \cap F)^{(\beta)}
	                       = \bigcap_{\beta < \lambda} (K^{(\beta)} \cap F)
												 = \bigcap_{\beta < \lambda} K^{(\beta)} \cap F 
												 = K^{(\lambda)} \cap F.
\end{equation*}
\end{itemize}
This concludes the proof.
\end{proof}
\begin{Lemma}\label{Lem4}
Let $\alpha$ be a countable ordinal number such that $\alpha>0$.  
Let $p\in\omega \smallsetminus \{0\}$.  Suppose that
for all $\widetilde{K}\in\mathcal{K}$ such that 
$\mathcal{CB}(\widetilde{K})=(\alpha,1)$, there exists a 
homeomorphism of $\widetilde{K}$ onto $\omega^\alpha+1$. 
Then, for all $K\in\mathcal{K}$ such that 
$\mathcal{CB}(K)=(\alpha,p)$, there exists a homeomorphism 
of $K$ onto $\omega^\alpha\cdot p+1$.
\end{Lemma}
\begin{proof}
Let $K\in\mathcal{K}$ be such that 
$\mathcal{CB}(K)=(\alpha,p) \in  \Omega \times \omega$. We write 
$K^{(\alpha)}=\{x_1,x_2,\ldots,x_p\}$, where $x_i < x_j$, for 
all $i, j \in I:= \{1, \ldots, p\}$ with $i < j$. We see that 
for all $k \in \{1,\ldots,p-1\}$, there exists  
$z_k \in (x_k, x_{k+1})$ such that $z_k \notin K$. We now consider 
the sets
\begin{align}\label{eq:sets1}
  K_1  &= K\cap (-\infty,z_1], \nonumber\\
  K_{k}&= K\cap [z_{k-1},z_{k}], 
				  \; \; k \in \{2, \ldots, p-1\}, \nonumber \\
	K_p  &= K\cap [z_{p-1}, +\infty).			
\end{align}
Proceeding as in the proof of Lemma \ref{Lem3}, 
it is possible to show that the finite family $(K_k)_{k\in I}$
satisfies the following properties:
\begin{itemize}
  \item $K_k\subset K$, for all $k\in I$.
  \item $K_k\in \mathcal{K}$, for all $k\in I$.
  \item $x_k\in K_k' \neq \varnothing$, for all $k\in I$.
  \item $(K_k)_{k\in I}$ is a pairwise disjoint finite 
	      sequence in $\mathcal{K}$.
  \item $\displaystyle \biguplus_{k\in I} K_k=K$.
\end{itemize}
By using Lemma \ref{Lemtech}, we have that for all $k\in I$, 
$K_k^{(\alpha)}=\{x_k\}$. Therefore, for all $k\in I$, 
$\mathcal{CB}(K_k)=(\alpha,1)$. Thus, for all $k \in I$, there 
exists a homeomorphism $f_k$ of $K_k$ onto $\omega^\alpha+1$. We now 
define the function $f$ given by
\begin{equation*}
    \funcion{f}{K}{\tau+1}{z}
    {f(z)= 
		\begin{cases} 
		f_1(z), & \text{if } z \in K_1, \\[2mm]
		\displaystyle \sum_{j=1}^{k-1}\omega^{\alpha}+1+f_k(z), &  
    \text{if }z\in K_k, \text{ for some } k \in I \smallsetminus \{1\},
		\end{cases}}
\end{equation*}
where
\begin{equation*}
    \tau:=\sum_{j=1}^{p}\omega^{\alpha}=\omega^{\alpha} \cdot 
		\sum_{j=1}^{p} 1=\omega^{\alpha}\cdot p.
\end{equation*}
Proceeding in a similar fashion as in the items (a), (b) and (c) in the 
proof of Lemma \ref{Lem3}, we obtain that $f$ is a homeomorphism of 
$K$ onto $\omega^\alpha\cdot p+1$.
\end{proof}
\begin{Lemma}\label{Lem5}
Let $\alpha$ be a countable ordinal number such that $\alpha>0$.  
Let $p\in\omega \smallsetminus \{0\}$.  Then, for all $K\in\mathcal{K}$ 
such that $\mathcal{CB}(K)=(\alpha,p)$, there exists a homeomorphism 
of $K$ onto $\omega^\alpha\cdot p+1$.
\end{Lemma}
\begin{proof}
We will use Strong Transfinite Induction.  By Lemmas~\ref{Lem2} 
and~\ref{Lem4}, the result holds for $\alpha=1$. We now consider 
$\alpha \in \Omega$ such that $\alpha>1$, and we suppose that 
the result is true for all ordinal number $\beta$ such that 
$0<\beta<\alpha$. Lemmas~\ref{Lem3} and~\ref{Lem4} imply the 
result for $\alpha$. Hence, the lemma is proved.
\end{proof}
Next result contains the reciprocal of Theorem~\ref{Th2}. 
\begin{Theorem}\label{Th3}
If $K_1,K_2\in \mathcal{K}$ and $\mathcal{CB}(K_1)=\mathcal{CB}(K_2)$, 
then $K_1\sim K_2$.
\end{Theorem}
\begin{proof}
If $\mathcal{CB}(K_1)=\mathcal{CB}(K_2)=(0,p) \in \Omega \times \omega$, 
we get $|K_1|=|K_2|=p$.  Then, $K_1\sim K_2$. \\
We now suppose that $\mathcal{CB}(K_1)=\mathcal{CB}(K_2)=(\alpha,p)$, 
with $\alpha>0$.  By Proposition~\ref{Prop1}, 
$p \in \omega \smallsetminus \{0\}$. By Lemma~\ref{Lem5}, there exist two  
homeomorphisms, $g$ of $K_1$ onto $\omega^\alpha\cdot p+1$ 
and $h$ of $K_2$ onto $\omega^\alpha\cdot p+1$. Therefore,
$ {f=h^{-1}\circ g}\colon{K_1}\longmapsto{K_2}$ is a homeomorphism of 
$K_1$ onto $K_2$. Hence, $K_1 \sim K_2$.
\end{proof}
Theorems~\ref{Th2} and \ref{Th3} fully characterize the partition of 
$\mathcal{K}$ by the Cantor-Bendixson characteristic.
\subsection{Cardinality of the set $\mathscr{K}$}
Combining the previous results we obtain the cardinality of $\mathscr{K}$.
\begin{Theorem}\label{Thcardinality}
The set $\mathscr{K}$, given by \eqref{partition}, has cardinality $\aleph_1$.
\end{Theorem}
\begin{proof}
We define the function 
\begin{equation}\label{eq:CBtilde}
    \funcion{\mathcal{\widetilde{CB}}}
		{\mathscr{K}}{\big(\Omega\times(\omega\smallsetminus\{0\})\big)\cup(0,0)}
		{[K]}{\mathcal{\widetilde{CB}}([K])=\mathcal{CB}(K)=(\alpha,p).}
\end{equation}
By Theorem~\ref{Th2} and Proposition~\ref{Prop1}, we see that 
$\mathcal{\widetilde{CB}}$ is well-defined. Moreover, Corollary~\ref{Cor1} 
implies that $\mathcal{\widetilde{CB}}$ is a surjective function. Furthermore, by 
Theorem~\ref{Th3}, $\mathcal{\widetilde{CB}}$ is an injective function.  
Then, 
\begin{equation*}
    |\mathscr{K}|=|\big(\Omega\times(\omega\smallsetminus\{0\})\big)\cup(0,0)|
		= |\Omega\times\omega|=|\Omega|=\aleph_1.\qedhere
\end{equation*}
\end{proof}
Last theorem shows that 
\begin{equation*}
    \aleph_0< \aleph_1 = |\mathscr{K}|\leq 2^{\aleph_0} = \mathfrak{c},
\end{equation*}
where $\mathfrak{c}$ is the cardinality of $\mathbb{R}$.
\subsection{A ``primitive'' related to the Cantor-Bendixson derivative 
of compact subsets of the real line}
We end this paper with a last theorem that we can view as 
a generalization of Theorem \ref{Th1} and Corollary \ref{Cor1} 
given in Section \ref{Sectfam}.  The next result shows that 
for any compact subset of the reals, there is a primitive-like 
set associated to its Cantor-Bendixson derivative.
\begin{Theorem}\label{Thprimitive}
Suppose that $\alpha\in\Omega$. Let $F$ be a compact subset of $\mathbb{R}$. 
Then, there exists a compact set $\mathcal{F} \subset \mathbb{R}$ 
such that $\mathcal{F}^{(\alpha)}=F$.
\end{Theorem}
\begin{proof}
If $\alpha=0$, we define $\mathcal{F}=F$ and the result holds. \\
From now on, we suppose that $\alpha>0$.  There are two cases. First, if $F$ 
is perfect, i.e. $F=F'$, we can take $\mathcal{F} = F$, and the result 
follows.\\
We now assume that $F\neq F'$.  Since $F\smallsetminus F'$ is the 
set of all isolated points of $F$, we have that 
$F\smallsetminus F' \neq \varnothing$ is countable. 
Hence, $F\smallsetminus F'=\{x_n:n\in I\}$, where 
$\varnothing \neq I\subset \omega$, and $x_n\neq x_m$, for all 
$n,m\in I$ with $n\neq m$.  Furthermore, for all $n\in I$, there 
exists $ r_n\in (0,\frac{1}{n+1})$ such that 
$(x_n- r_n,x_n+ r_n)\cap F=\{x_n\}$. By Theorem~\ref{Th1}, we see that 
for all $n \in I$, there exits $K_n\in\mathcal{K}$ such that 
$K_n\subset (x_n- r_n,x_n]$ and $K_n^{(\alpha)}=\{x_n\}$. Since 
$\left( {(x_n-r_n, x_n]} \right)_{n \in I}$ is a pairwise disjoint 
sequence of intervals, we see that $(K_n)_{n \in I}$ is a 
pairwise disjoint sequence in $\mathcal{K}$.  We now define the set 
$\mathcal{F} \subset \mathbb{R}$ given by
\begin{equation}\label{eq:primitive1}
    \mathcal{F}:= \biguplus_{n\in I} K_n \cup F.
\end{equation}
\paragraph{\underline{Claim 1}} $\mathcal{F}$ is a compact subset of $\mathbb{R}$. \\
In fact, let $(z_k)_{k\in\omega}$ be a sequence in $\mathcal{F}$ 
such that $z_k\to z \in \mathbb{R}$ when $k\to+\infty$.  There are three 
cases.
\begin{enumerate}[label=(\roman*),leftmargin=*]
\item 
If $\{k\in\omega: z_k\in F\}$ is infinite, there exists a subsequence  
$(z_{\phi(k)})_{k\in\omega}$ in $F$, where $\phi:\omega \longmapsto \omega$ is a 
strictly increasing function.  Since $F$ is closed, we conclude that 
$z\in F\subset \mathcal{F}$.
\item 
We now suppose that there exists $m\in I$ such that 
$\{k\in\omega: z_k\in K_m\}$ is infinite. Similarly as in the previous 
case, we obtain that $z\in K_m \subset \mathcal{F}$.
\item 
Finally, we assume that for all $n\in I$,  $\{k\in\omega: z_k\in K_n\}$ 
is a finite set and $\{k\in\omega: z_k\in F\}$ is also finite. Thus, there exists 
a subsequence $(z_{\psi(k)})_{k\in\omega}$, where $\psi:\omega \longmapsto \omega$ 
is a strictly increasing function, and there is also a strictly increasing function 
$\sigma: \omega \longmapsto I$ such that for all $k\in\omega$
\begin{equation}\label{eq:primitive2}
	z_{\psi(k)}\in K_{\sigma(k)}\subset (x_{\sigma(k)}-r_{\sigma(k)},x_{\sigma(k)}].
\end{equation}
In order to prove the last assertion, we see that there exists $n_0 \in I$ such that 
$\{k \in \omega : z_k \in K_{n_0}\} \neq\varnothing$.  Then, there is 
$k_0 \in \omega$ with $z_{k_0} \in K_{n_0}$.  We thus define 
$\psi(0):=k_0$ and $\sigma(0):=n_0$.  We now get $n_1 \in I$ with $n_1>n_0$ and 
such that $\{k \in \omega: z_k \in K_{n_1}, k> k_0\} \neq \varnothing$.  So, there 
exists $k_1 \in \omega$ with $k_1>k_0$ and such that $z_{k_1} \in K_{n_1}$.  We define  
$\psi(1):=k_1$ and $\sigma(1):=n_1$. By continuing this process, functions 
$\psi$ and $\sigma$ are recursively obtained.  From~\eqref{eq:primitive2}, we have that  
for all $k \in \omega$, $|x_{\sigma(k)}-z_{\psi(k)}|<r_{\sigma(k)}<\frac{1}{\sigma(k)+1}$. 
As $(z_{\psi(k)})_{k\in\omega}$ converges to $z$, it follows that 
$(x_{\sigma(k)})_{k\in\omega}$ also converges to $z$. Since, the elements of the last 
sequence belong to $F$, and $F$ is closed, we conclude that $z\in F\subset\mathcal{F}$.
\end{enumerate}
From (i), (ii) and (iii), $\mathcal{F}$ is a closed subset of $\mathbb{R}$.  Moreover, 
since $F$ is bounded, there exist $a,b\in\mathbb{R}$, with $a<b$, 
such that $F\subset [a,b]$. Then, $\mathcal{F}\subset [a-1,b]$, i.e.,
$\mathcal{F}$ is bounded. Hence, $\mathcal{F}$ is a compact subset 
of $\mathbb{R}$. \\
\paragraph{\underline{Claim 2}} $\mathcal{F}^{(\alpha)} = F$.  \\
Actually, we will show  the following more general result: 
for all countable ordinal number $\beta\in\Omega$ such that $\beta\le \alpha$  
\begin{equation}\label{eq:primitive3}
    \mathcal{F}^{(\beta)}= \biguplus_{n\in I} K_n^{(\beta)}\cup F.
\end{equation}
In order to prove~\eqref{eq:primitive3}, we proceed by Transfinite Induction 
as in Theorem~\ref{Th1}.
\begin{enumerate}[label=(\alph*),leftmargin=*]
\item 
If $\beta=0$, then the result holds immediately.
\item 
We now suppose that~\eqref{eq:primitive3} is true for a given $\beta \in \Omega$ 
such that $\beta<\alpha$. We note that for all $n \in I$, 
$K_n^{(\beta+1)}\subset \mathcal{F}^{(\beta+1)}$.  Then,
\begin{equation*}
	\biguplus_{n\in I} K_n^{(\beta+1)}\subset \mathcal{F}^{(\beta+1)}.
\end{equation*}
Furthermore, by the induction hypothesis, $F \subset \mathcal{F}^{(\beta)}$.  Then,   
$F' \subset \mathcal{F}^{(\beta+1)}$. Moreover, 
\begin{equation*}
	F\smallsetminus F' = \biguplus_{n \in I} \{x_n\}   
	  = \biguplus_{n\in I} K_n^{(\alpha)}
		\subset \biguplus_{n\in I} K_n^{(\beta+1)}
		\subset \mathcal{F}^{(\beta+1)}.
\end{equation*}
Hence,
\begin{equation}\label{eq:primitive4}
	\biguplus_{n\in I}K_n^{(\beta+1)} \cup F \subset \mathcal{F}^{(\beta+1)}.
\end{equation}
In order to show the reverse inclusion, we take $x\in \mathcal{F}^{(\beta+1)}$. 
Using the induction hypothesis, we see that
\begin{equation*}
	x\in \mathcal{F}^{(\beta+1)} = (\mathcal{F}^{(\beta)})'
		=\left(\biguplus_{n\in I}K_n^{(\beta)} \cup F \right)'
		= \left(\biguplus_{n\in I}K_n^{(\beta)} \right)' \cup F'.
\end{equation*}
Using now Claim 1, we have that $\mathcal{F}$ is closed.  Then, 
\begin{equation*}
	x\in \mathcal{F}^{(\beta+1)} \subset \mathcal{F}^{(\beta)}
		=\biguplus_{n\in I}K_n^{(\beta)} \cup F.
\end{equation*}
If $x\in F$, there is nothing left to show. On the other hand, if $x\not\in F$, 
there exists $m\in I$ such that $x\in K_m^{(\beta)}\subset (x_m-r_m,x_m]$. We now 
assume, by contradiction, that $x\notin K_m^{(\beta+1)}$. Then, $x$ is an isolated 
point of $K_m^{(\beta)}$. Since $x\neq x_m \in F$, there is 
$0<\varepsilon<\min\{x-x_m+r_m, x_m-x\}$ such that 
\begin{equation*}
	 (x-\varepsilon,x+\varepsilon)\cap K_m^{(\beta)} = \{x\}.
\end{equation*}
Moreover, as $(x-\varepsilon,x+\varepsilon)\subset (x_m-r_m,x_m)$, we conclude that  
for all $n \in I$ with $n\neq m$, 
\begin{equation*}
	(x-\varepsilon,x+\varepsilon) \cap  K_n^{(\beta)} = \varnothing.
\end{equation*}
Then, 
\begin{equation*}
	(x-\varepsilon,x+\varepsilon)\cap \biguplus_{n\in I} K_n^{(\beta)} = \{x\}.
\end{equation*}
Therefore, $x$ is an isolated point of $\biguplus_{n\in I} K_n^{(\beta)}$.  Since 
$x\not\in F$, and $F$ is closed, we see that $x\not\in F'$.  Hence, 
$x \in \left(\biguplus_{n\in I} K_n^{(\beta)}\right)'$, which is contradictory. 
In consequence,  
\begin{equation*}
	x\in K_m^{(\beta+1)} \subset \biguplus_{n\in I}K_n^{(\beta+1)}.
\end{equation*}
Thus, summarizing, we can conclude that 
\begin{equation}\label{eq:primitive5}
	\mathcal{F}^{(\beta+1)} \subset \biguplus_{n\in I}K_n^{(\beta+1)} \cup F.
\end{equation}
From~\eqref{eq:primitive4} and~\eqref{eq:primitive5}, we get
\begin{equation*}
	\mathcal{F}^{(\beta+1)} = \biguplus_{n\in I}K_n^{(\beta+1)} \cup F.
\end{equation*}
\item 
Finally, let $\gamma \neq 0$ be a limit ordinal such that $\gamma \le \alpha$ and 
we assume that for all ordinal number $\delta$ such that $\delta<\gamma$, 
\begin{equation} \label{eq:primitive6}
    \mathcal{F}^{(\delta)}= \biguplus_{n\in I}K_n^{(\delta)} \cup F. 
\end{equation}
Using~\eqref{eq:primitive6}, we obtain  
\begin{align} 
   \biguplus_{n\in I}K_n^{(\gamma)}\cup F
   &=  \biguplus_{n\in I}\left(\bigcap_{\delta<\gamma}
		   K_n^{(\delta)}\right)\cup F\notag\\
   &\subset  \bigcap_{\delta<\gamma}\left(\biguplus_{n\in I}
		   K_n^{(\delta)}\right)\cup F\notag\\
   &=  \bigcap_{\delta<\gamma}\left(\biguplus_{n\in I}
		   K_n^{(\delta)}\cup F\right)\notag\\
   &=\bigcap_{\delta<\gamma}\mathcal{F}^{(\delta)}\notag\\
   &= \mathcal{F}^{(\gamma)}.\label{eq:primitive7}
\end{align}
In order to show the other inclusion, we take $x\in \mathcal{F}^{(\gamma)}$. Using 
the induction hypothesis~\eqref{eq:primitive6}, we see that
\begin{equation*}
   \mathcal{F}^{(\gamma)} = \bigcap_{\delta<\gamma}  \mathcal{F}^{(\delta)} = 
	\bigcap_{\delta<\gamma}\left(\biguplus_{n\in I}K_n^{(\delta)}
	\cup F\right).
\end{equation*}
Then, either $x\in F$ or for all ordinal number $\delta$ such that 
$\delta<\gamma$, there exists $n\in I$ such that 
$x\in K_n^{(\delta)}$. If $x\in F$, then there is nothing else to be done.  If 
$x\notin F$, there  is $N\in I$ such that $x\in K_N^{(0)}=K_N$.  We now assume, 
to get a contradiction, that there is an ordinal number $\delta_0$ with 
$\delta_0<\gamma$ and such that $x \notin K_N^{(\delta_0)}$.  Since there is 
$l\in I$ with $l \neq N$ such that $x\in K_l^{(\delta_0)} \subset K_l$, we 
obtain a contradiction with the fact that $K_l \cap K_N = \varnothing$.  Hence,   
for all ordinal number $\delta$ such that $\delta<\gamma$, $x \in K_N^{(\delta)}$. 
In consequence,   
\begin{equation*}
    x\in\bigcap_{\delta<\gamma}K_N^{(\delta)}=K_N^{(\gamma)} 
		\subset \biguplus_{n\in I}K_n^{(\gamma)}.
\end{equation*}
Thus,
\begin{equation} \label{eq:primitive8}
   \mathcal{F}^{(\gamma)}\subset\biguplus_{n\in I}K_n^{(\gamma)} \cup F.
\end{equation}
From~\eqref{eq:primitive7} and~\eqref{eq:primitive8}, we have that   
\begin{equation*}
   \mathcal{F}^{(\gamma)}=\biguplus_{n\in I}K_n^{(\gamma)}\cup F.
\end{equation*}
\end{enumerate}
By (a), (b) and (c), we obtain~\eqref{eq:primitive3} for all countable 
ordinal number $\beta$ such that $\beta\le\alpha$.  Finally, 
using~\eqref{eq:primitive3} with $\alpha$, we get  
\begin{equation*}
    \mathcal{F}^{(\alpha)}=\biguplus_{n\in I} 
		K_n^{(\alpha)}\cup F= \biguplus_{n\in I} 
		\{x_n\}\cup F=F,
\end{equation*}
which finishes the proof. \qedhere
\end{proof}
\bibliographystyle{siam}
\bibliography{Cantor-Bendixson-arXiv}

\begin{thebibliography}{1}

\bibitem{Cantor1872}
{\sc G.~Cantor}, {\em {Ueber die Ausdehnung eines Satzes aus der Theorie der
  trigonometrischen Reihen}}, Math. Ann., {5} (1872), pp.~123--132.

\bibitem{Cantor1879}
\leavevmode\vrule height 2pt depth -1.6pt width 23pt, {\em {Ueber unendliche,
  lineare Punktmannichfaltigkeiten I}}, Math. Ann., {15} (1879), pp.~1--7.

\bibitem{Cantor1880}
\leavevmode\vrule height 2pt depth -1.6pt width 23pt, {\em {Ueber unendliche,
  lineare Punktmannichfaltigkeiten II}}, Math. Ann., {17} (1880), pp.~355--358.

\bibitem{Cantor1882}
\leavevmode\vrule height 2pt depth -1.6pt width 23pt, {\em {Ueber unendliche,
  lineare Punktmannichfaltigkeiten III}}, Math. Ann., {20} (1882),
  pp.~113--121.

\bibitem{Cantor1883s}
\leavevmode\vrule height 2pt depth -1.6pt width 23pt, {\em {Sur divers
  th\'eor\`emes de la th\'eorie des ensembles de points situ\'es dans un espace
  continu \`a $n$ dimensions}}, Acta Math., {2} (1883), pp.~409--414.

\bibitem{Cantor1883}
\leavevmode\vrule height 2pt depth -1.6pt width 23pt, {\em {Ueber unendliche,
  lineare Punktmannichfaltigkeiten IV}}, Math. Ann., {21} (1883), pp.~51--58.

\bibitem{Sierpinski}
{\sc S.~Mazurkiewicz and W.~Sierpinski}, {\em {Contribution \`a la topologie
  des ensembles d\'enombrables}}, Fund. Math., {1} (1920), pp.~17--27.

\end{thebibliography}

\end{document}